\documentclass{amsart}
\usepackage{amsmath, amssymb, latexsym,  amsthm, amscd, xy, color}

\newcommand{ \PP }{{\mathbb P}}
\newcommand{ \CC }{{\mathbb C}}

\newcommand{ \ii }{{\mathcal I}}
\newcommand{ \oo }{{\mathcal O}}

\newtheorem{thm}{Theorem}[section]
 \newtheorem{cor}[thm]{Corollary}
 \newtheorem{lem}[thm]{Lemma}
 \newtheorem{prop}[thm]{Proposition}

 \theoremstyle{definition}

\newtheorem*{conj*}{Conjecture}

\newtheorem*{thm*}{Theorem}

 \theoremstyle{remark}
 \newtheorem{rem}[thm]{Remark}

  \newtheorem{nota}[thm]{Notation}

\newcommand{\rk}{\operatorname{rk}}

\begin{document}

\title{Progress on the symmetric Strassen conjecture}

\author[E. Carlini]{Enrico Carlini}
\address[E. Carlini]{Department of Mathematical Sciences, Politecnico di Torino, Turin, Italy}
\email{enrico.carlini@polito.it}
\address[E.Carlini]{School of Mathematical Sciences, Monash University, Clayton, Australia.}
\email{enrico.carlini@monash.edu}

\author[M.V. Catalisano]{Maria Virginia Catalisano}
\address[M.V. Catalisano]{DIME - Dipartimento di ingegneria meccanica, energetica, gestionale e dei trasporti, University of Genoa, Italy.}
\email{catalisano@diptem.unige.it}

\author[L. Chiantini]{Luca Chiantini}
\address[L. Chiantini]{Department of Information Engineering and Mathematics, University of Siena,Italy.}
\email{luca.chiantini@unisi.it}


\begin{abstract}
Let $F$ and $G$ be homogeneous polynomials in disjoint sets of variables. We prove that the Waring rank is additive, thus proving the symmetric Strassen conjecture, when either $F$ or $G$ is a power, or $F$ and $G$ have two variables, or either $F$ or $G$ has small rank.
\end{abstract}


 \maketitle

\thispagestyle{empty}

\section{Introduction}

In his famous result of 1969 Strassen showed that it is possible to multiply two $2\times 2$ matrices using seven basic operations rather than eight, see \cite{landsbergstrassen} for more on this. Using this fact, a better algorithm was produced to multiply matrices of any size and this was proved to have the best possible computational complexity by Winograd in \cite{Winograd}.

After Strassen's result, it was clear that even straightforward procedures can require fewer operations than expected.  In \cite{strassenconj} Strassen formulated his well known {\em additive conjecture} for bilinear maps: {\em Given bilinear maps $\phi,\psi$ and two pairs of matrices $A,B$, and $C,D$  the computational complexity of simultaneously computing $\phi(A,B)$ and $\psi(C,D)$ is the sum of the complexities of $\phi$ and $\psi$.} The conjecture stands open since its formulation in 1973, for some partial results see \cite{winograddirectsumconj}.

Strassen conjecture can be naturally stated in terms of tensors and the notion of tensor rank, see \cite{landsbergbook}. Note that an analogue of the additive conjecture for 
{\it approximate} complexity (border rank, in more recent terminology) does not hold (see \cite{Sch}).

We will focus on the relevant case of symmetric tensors, that is, the case of homogeneous polynomials, also known as forms. The {\em rank}, or {\em Waring rank}, of a degree $d$ form $F$ is $$\rk(F)=\min\{r : F=L_1^d+\ldots+L_r^d\mbox{ for linear forms }L_i\}.$$
We can now state the symmetric version of the Strassen Additivity Conjecture (SAC)
\begin{conj*} (SAC). Let $d>1$. If $F\in\mathbb{C}[x_0,\ldots,x_n]$ and $G\in\mathbb{C}[y_0,\ldots,y_m]$ are non-zero forms of degree $d$, then
$$ \rk(F+G) =\rk(F) + \rk (G). $$
\end{conj*}
Also the symmetric version of the conjecture stands open.
A relevant contribution to its study is the 2012 paper \cite{carcatgermonomi} where SAC is proved for the sum of (several) monomials.

In this paper we ontribute to the study of SAC by proving that the conjecture holds for one and two variables and for forms of small rank. More precisely, we prove the following theorem.

\begin{thm}\label{mainthm} Let $d>1$ and let $F\in\mathbb{C}[x_0,\ldots,x_n]$, $G\in\mathbb{C}[y_0,\ldots,y_m]$ be non-zero degree $d$ forms. If
$n=0\mbox{ or } m=0\mbox{ or } n,m=1$,
then $\rk(F+G) =\rk(F) + \rk (G)$.
\end{thm}
In Corollary \ref{Marvi2} we show that SAC holds even with no assumptions on $n$ and $m$ if one of the forms satisfies a condition on the rank (see the statement of the corollary).

We obtain Theorem \ref{mainthm}  using algebraic geometry. In particular, we exploit a deep knowledge of the Hilbert functions of finite sets of points. A similar approach was considered also in \cite{BaBe2012}.
We will use the strong structural result contained
in \cite{BGM}. Our techniques can be extended
 to study SAC in more than two variables, but we need some genericity assumptions on the forms $F$ and $G$; we are currently investigating this approach.

\section{Preliminaries}

In this section we state our notation and we recall the main results that we will use.

\begin{nota}
For a zero-dimensional scheme $Z\subset\PP^N$, we will denote with $h_Z$  the
Hilbert function of $Z$,  that is the function
which associates to any integer $n$ the dimension of the image 
of the map $H^0\oo_{\PP^N}(n)\to H^0\oo_Z(n)$. We will denote with $Dh_Z$ the first difference of the Hilbert function, that is $Dh_Z(n)=Dh_Z(n+1)-Dh_Z(n)$.
\end{nota}

\begin{nota}
In this paper, $F$ and $G$ will denote
forms of degree $d$, with 
 $d>1$,  in disjoint sets of variables. Namely, $F\in\mathbb{C}[x_0,\ldots,x_n]$ and $G\in\mathbb{C}[y_0,\ldots,y_m]$. 
\end{nota}

We will use the following formulation of SAC. Let $\rk(F)=r$, $\rk(G)=s$, and write $$F=F_1^d+\dots + F_r^d\mbox{ and }G=G_1^d+\dots +G_s^d,$$ for linear forms $F_i\in\mathbb{C}[x_0,\ldots,x_n] $ and $G_i\in\mathbb{C}[y_0,\ldots,y_m]$. With these notation SAC is equivalent to the fact that it is not possible to have
\begin{equation}\label{SAC}
 F_1^d+\dots + F_r^d +  G_1^d+\dots +G_s^d\quad = \quad H_1^d+\dots + H_t^d
 \end{equation}
where $t<r+s$ and the $H_i$ are linear forms in $\mathbb{C}[x_0,\ldots,x_n,y_0,\ldots,y_m]$.

We begin with a reduction step.

\begin{lem}\label{noH}
To prove  SAC as stated in  \eqref{SAC},  it is enough to consider
\[H_i\not\in \mathbb{C}[x_0,\ldots,x_n]\cup \mathbb{C}[y_0,\ldots,y_m] .\]
\end{lem}
\begin{proof}
Assume that we have $F,G$ as in \eqref{SAC}, 
with $a$ of the $H_i$'s (say the first $a$ of them)
belonging to $\CC[x_0,\dots,x_n]$, and $b$ of the $H_i$'s (say the last $b$)
belonging to $\CC[y_0,\dots,y_m]$.

Notice that $F'= F_1^d+\dots + F_r^d- H_1^d-\dots - H_a^d$ is a form of degree $d$
in $\CC[x_0,\dots,x_n]$, whose rank is at least $r-a$. Otherwise
we could write $F$ as a form  of rank less than $r-a$ plus the sum of $a$ powers of linear forms,
contradicting the fact that $\rk(F)=r$. 

Similarly, $G'= G_1^d+\dots + G_s^d- H_{t-b+1}^d-\dots - H_t^d$ is a form of degree $d$
in $\CC[y_0,\dots,y_m]$, whose rank is at least $s-b$.

Then  $F'+G' = H_{a+1}^d+ \dots +H_{t-b+1}^d$, with $\rk(F')\geq r-a$, $\rk(G')\geq s-b$. The 
remaining forms  $H_i$ are $t-a-b<r-a+s-b$ and none of them lies either in $\CC[x_0,\dots,x_n]$
or in $\CC[y_0,\dots,y_m]$. Assuming that SAC holds in this situation yields a contradiction, thus proving the conjecture for any choice of the forms $H_i$.
\end{proof}

We recall some basic properties of binary forms, that is forms in two variables.
\begin{rem}\label{binaryREM}
The linear span of a rational normal curve of degree $d$ contains all binary forms in a given set of variables. Hence, any set of $i\geq d+1$ degree $d$ binary forms is linearly dependent. Moreover, any degree $d$ binary form has rank at most $d$, see \cite{comasseiguer} for more details.
\end{rem}

\section{One variable}

Assume that either $F$ or $G$ is a polynomial in one variable. In this
case, a direct computation proves that SAC holds.

\begin{prop}\label{Marvi} If $n=0$ or $m=0$, then SAC holds. 
\end{prop}
\begin{proof}

We will prove the statement for $m=0$. Let $F(x_0,\ldots,x_n)\neq 0$ be a degree 
$d>1$ form in $ k[x_0 ,\ldots, x_n]$ and let $G(y_0)=y_0^d$. We set
\[ P(x_0,\ldots,x_n,y_0)= F(x_0,\ldots,x_n) +G(y_0),\]
and we will show that $\rk(P)=\rk(F)+\rk(G)=\rk(F)+1$.

Let $\mathrm{rk}(F)=r$.  For $r=1$, the conclusion  immediately follows, 
so assume $ r >1$. Clearly  $\mathrm{rk}(P) \leq r+1$. If  $\mathrm{rk}(P) <r$, since
$ F(x_0,\ldots,x_n) = P(x_0,\ldots,x_n, 0)$, we get a contradiction.

We now prove that it is not possible to have $\rk(P)=r$. Let
$$P = (\alpha_1 y_0 +L_1)^d+\dots + (\alpha_r y_0 +L_r)^d,
$$
where the $L_i$ are linear forms in $x_0,\ldots,x_n.$
By setting $y_0=0$ we have that 
$$F = L_1^d+\dots + L_r^d,
$$
and thus
$$L_1^d+\dots + L_r^d + y_0^d =  (\alpha_1 y_0 +L_1)^d+ (\alpha_2 y_0 +L_2)^d+\dots 
+ (\alpha_r y_0 +L_r)^d.
$$
Since $G\neq 0$, the $\alpha_i$'s cannot all be equal to zero. Assume that $\alpha_1\neq 0. $
Let $\gamma$ be a  d-th root of the identity. Since $d>1$, we may assume that 
$\gamma \neq \alpha_1$.

Now set 
$$y_0 = \frac {L_1 } {\gamma-\alpha_1},$$ 
so that 
$$(\alpha_1 y_0 +L_1)^d= \left ( \frac {\gamma L_1 } {\gamma-\alpha_1}\right )^d =
\left ( \frac { L_1 } {\gamma-\alpha_1}\right )^d = y_0^d.
$$
Hence we get 
$$F =  \left(\alpha_2 \left ( \frac {L_1 } {\gamma-\alpha_1} \right )+L_2\right)^d+\dots +
 \left ( \alpha_r \left ( \frac {L_1 } {\gamma-\alpha_1} \right ) +L_r\right )^d.
$$
It follows that $\mathrm{rk}(F) <r$, a contradiction. Hence $\mathrm{rk}(P) =r+1.$
\end{proof}

To extend this result, we use apolarity theory, see \cite{Ge,IaKa}. Recall that $(F^\perp)_1$ is the vector space of linear differential operator vanishing 
on a form $F$. Following \cite{reducing} and \cite{Birula} we say that $F\in\mathbb{C}[x_0,\ldots,x_n]$ 
{\em essentially involves} $N$ variables if
\[\dim (F^\perp)_1=n+1-N.\]
In other words, $F$ essentially involves $N$ variables if there exists a linear change of variables such that $F\in\mathbb{C}[z_0,\ldots,z_{N-1}]\subset\mathbb{C}[x_0,\ldots,x_n]$.

\begin{cor}\label{Marvi2} If  $F$ essentially involves $\rk(F)$ 
variables, or $G$ essentially involves $\rk(G)$ variables,  
then SAC holds.
\end{cor}
\begin{proof} We let $r=\rk(F),s=\rk(G)$ and we prove the statement in the case that $G$ essentially involves $s$ variables. After a change of variables we may assume 
$$G = y_0^d+\dots +y_{s-1}^d.$$
Arguing by induction on $s$, we see that 
$F+y_0^d+\dots +y_{s-2}^d$ has rank $r+s-1$, thus by Proposition \ref{Marvi},
$F+G=(F+y_0^d+\dots +y_{s-2}^d)+y_{s-1}^d$ has rank $(r+s-1)+1=r+s$.
\end{proof}

\section{SAC and the Hilbert function}

Here we introduce our algebro-geometric approach to SAC with the relative notation.

\begin{nota} 
As above let $F=F_1^d+\dots+F^d_r$, $r=\rk(F)$ and $G=G_1^d+\dots+G^d_s$
, $s=\rk(G)$, where the forms $F_i$ and the $G_i$ have degree one. We introduce the sets  
$$\begin{matrix} Z(F)=\{[F_1],\dots,[F_r]\}\subset\PP^n=\PP(\CC[x_0,\ldots,x_n]),  \\ Z(G)=\{[G_1],\dots,[G_s]\}\subset\PP^m=\PP(\CC[y_0,\ldots,y_m]).
\end{matrix}$$	
We will argue by contradiction assuming that
\begin{equation}\label{theeq}
F+G=H_1^d+\dots+H_t^d,
\end{equation}
for $t<r+s$ and linear forms $H_i$. Finally, we set 
$$ Z(H)=\{[H_1],\dots,[H_t]\}\subset\PP^{n+m+1}=\PP(\CC[x_0,\ldots,x_n,y_0,\ldots,y_m]),$$	
$$Z= Z(H)\cup Z(F)\cup Z(G)\subset\PP^{n+m+1}.$$ 	 
\end{nota}

We will use the following geometric remarks.

\begin{rem} The standard exact sequence of sheaves 
\[0\to \ii_Z(d) \to \oo_{\PP^N}(d) \to\oo_Z\to 0\]
yields
$$ 0\to H^0\ii_Z(d) \to H^0 \oo_{\PP^N}(d)\to H^0\oo_Z \to H^1\ii_Z(d)\to 0.$$
By the definitions, we have that $h_Z(d)=|Z|(=\dim H^0\oo_Z)$ if and only if $H^1\ii_Z(d)=0$. 
In this case,  we will say that {\it $Z$ is separated in degree $d$} or that {\it $Z$ 
imposes independent conditions to degree $d$ forms}.
\end{rem}

\begin{rem}\label{h1rem} The equality in \eqref{theeq} implies that the following are 
linearly dependent forms:
\[F_1^d,\ldots,F_r^d,G_1^d,\ldots,G_m^d,H_1^d,\ldots,H_t^d.\]
Thus, by Lemma 1 in \cite{BaBe2012}, $Z$ does not impose independent conditions 
to forms of degree $d$. If we denote with $\nu$ the $d$-th Veronese map, one of the points of $\nu(Z)$, say $F_i^d$, depends linearly
on the others. Thus, any linear form which vanishes on  $\nu(Z)\setminus \{F^d_i\}$ 
also vanishes at $F^d_i$. Hence, any form 
of degree $d$ vanishing on $Z\setminus \{F_i\}$
also vanishes at $F_i$. In other words, $Z$ is not separated in degree $d$, and
the ideal sheaf of $Z$ is such that
$$h^1\ii_Z(d)>0.$$
 \end{rem}

\begin{rem} Let $L_F\simeq \PP^n$ and $L_G \simeq \PP^m$, be 
two skew linear subspaces of $\PP^{n+m+1}$, containing  $Z(F)$ and $Z(G)$, respectively.
The following holds: if for some $i>0$
\[h^1\ii_{Z(F)}(i)=h^1\ii_{Z(G)}(i)=0\] , 
then \[h^1\ii_{Z(F)\cup Z(G)}(i)=0.\]
To see this, we argue as follows. If $P$ is a point of $Z(F)$,
consider a degree $i$ hypersurface $X \subset L_F$
that contains all the points of $Z(F)\setminus\{P\}$. The cone
on X with vertex $L_G$  is a hypersurface of degree $i$ in $\PP^{n+m+1}$. This cone separates $P$ 
from the other points of $Z(F)\cup Z(G)$. A similar argument works if $P\in Z(G)$.
Thus $Z(F)\cup Z(G)$ is separated in degree $i$.
\end{rem}

The next two lemmas will be useful in the following section.

\begin{lem}\label{celine} Let $W$ be a set of $w$ distinct points spanning a linear space
of dimension at least three.
If $w\leq 2u-2$ and $h^1\ii_W(u-2)>0$, then there exists a line that contains at least $u$ points of $W$.
\end{lem}
\begin{proof} Consider the first difference $Dh_W$ of the Hilbert function of $W$.
Since
\[h^1\ii_W(u-2)>0,\]
then $Dh_W(u-1)>0$, so that $Dh_W(i)>0$ for all $i=0,\dots, u-1$.
Moreover $Dh_W(1)\geq 3$, since the points span at least a space of dimension $3$.
	Since $w\leq 2u-2$, we cannot have $Dh_W(i)\geq 2$ for $i=2,\dots, u-2 $.
It follows that there are at least two values $2\leq i_1,i_2\leq u-1$ for which 
$Dh_W(i_1)=Dh_W(i_2)=1$. By Theorem 3.6 in \cite{BGM}  it follows that we may assume the two steps are
$i_1=u-2$ and $i_2=u-1$ and moreover $u$ points of $W$ are aligned.
\end{proof}

\begin{lem}\label{resid}
Let $W_1,W_2\subset\PP^n$ be finite sets of points such that $W_1$ is contained 
in a proper linear subspace $R$, and $W_2\cap R=\emptyset$. 
If for some integer $u$ one has:
\begin{enumerate}
\item\label{a} $h^1\ii_{W_2}(u-1)=0$ and
\item\label{b} the cardinality of $W_1$ is at most $u+1$,
\end{enumerate}
then $h^1\ii_W(u)=0$, where  $W=W_1\cup W_2$.
\end{lem}
\begin{proof}
If we let $w_i$ be the cardinality of $W_i$, then it is easy to see that
\[h^0\oo_{\PP^n}(u)-w_1-w_2\leq h^0\ii_W(u).\]
We now use Castelnuovo's inequality,  
see \cite{sundials} Section 3 for more details and \cite{AH95} 
Section 2 for a proof. This inequality yields
$$ h^0\ii_W(u)\leq h^0\ii_{Res}(u-1)+h^0\ii_{Tr}(u),$$
where $Tr$ is the trace of $W$ on a generic hyperplane 
$H\supseteq R\supset W_1$ and $Res$ is the residue of $W$ with respect to $H$. 
More precisely $\ii_{Tr}=\ii_{W,H}=\ii_{W_1}$ and $\ii_{Res}=\ii_W:\ii_H=\ii_{W_2}$, 
that is $Res=W_2$ and $Tr=W_1$.
From \eqref{b} we get that $h^0\ii_{Tr}(u)=h^0\oo_{\PP^{n-1}}(u)-w
_1$. From \eqref{a} we get that $h^0\ii_{Res}(u-1)=h^0\oo_{\PP^{n}}(u-1)-w
_2$. Hence, the result follows.

\end{proof}

\section{Forms in two variables}

In this section, we assume that $F,G$ are forms in two variables,
that is we assume $n=m=1$, and we prove that SAC holds
in this setting.
We let $r=\rk(F),s=\rk(G)$ and we assume that
$$
F_1^d+\dots + F_r^d +  G_1^d+\dots + G_s^d  =  H_1^d+\dots + H_t^d
$$
for $t<r+s$ and for linear forms $F_i\in\mathbb{C}[x_0,x_1]$ 
and $G_i\in\mathbb{C}[y_0,y_1]$, while the $H_i$ are linear 
forms involving all the four variables $y_0,y_1,x_0,$ and $x_1$ (see Lemma \ref{noH}).

\begin{rem} \label{step0}
By Proposition \ref{Marvi} and its Corollary \ref{Marvi2}, we may assume $r,s\geq 3$.
Moreover, as noted in Remark \ref{binaryREM} the rank of a form of degree
$d$ in two variables is at most $d$. So we have $r,s\leq d$.
\end{rem}

We treat the cases of large $r,s$  using the Apolarity Lemma, see \cite{IaKa}.

\begin{lem}\label{maxcas} If $r=s= d$, then SAC holds.
\end{lem}
\begin{proof} If $F$ is a binary degree $d$ form of rank $d$, then 
the Hilbert-Burch Theorem (see Theorem 20.15 in \cite{eisenbudbook}) yields that $F^\perp=(A,B)$ where 
$\deg(A)=2,\deg(B)=d$ and $A$ is a square. From this we see that 
$F=LM^{d-1}$ for some linear forms $L$ and $M$, thus $F$ is a monomial. 
The conclusion follows as the rank is additive on coprime monomials, 
see \cite{carcatgermonomi}. 
\end{proof}
From now on we assume that:
$$  s\leq r  \mbox{ and } s\leq d-1.$$

We now show that $Z(H)$ spans $\PP^3$ and that at most $d$ points of $Z(H)$ lie on  a line. 

\begin{lem}\label{spanP3} There is no plane containing $Z(H)$.
\end{lem}
\begin{proof}  Since $Z(H)\subset\PP(\CC[x_0,x_1,y_0,y_1])=\PP^3$, we use apolarity 
theory to describe the dual space ${\PP^3}^*$. Thus,  ${\PP^3}^*$ can be 
identified with the space of linear differential operators on the polynomial 
ring $\CC[x_0,x_1,y_0,y_1]$. We now argue by contradiction. If a plane containing 
$Z(H)$ exists, then the plane is the vanishing locus of some differential operator $\partial$. 
Thus $\partial$ vanishes on $Z(H)$, and thus $\partial(H)=\partial(F+G)=0$, that is 
 $\partial(F)=-\partial(G)$. Note that $\partial(F)\in\CC[x_0,x_1]$ 
and $\partial(G)\in\CC[y_0,y_1]$ are forms of degree $d-1>0$, 
hence $\partial(F)=\partial(G)=0$. The latter is impossible, since $F,G$ 
are binary forms of rank at least three.
\end{proof}

\begin{rem}\label{dline} If $t$ is the minimal integer such that $F+G=H_1^d+\dots + H_t^d$,
then we may assume that at most $d$ points of $Z(H)$ lie on a line.

Indeed, if $H_1,\dots, H_{d+1}$ lie on a line, then by Remark \ref{binaryREM} there is a linear
relation between their $d$-th powers, say 
$H_1^d= a_2H_2^d+\dots + a_{d+1}H_{d+1}^d$, so that
$$F+G = (1+a_2)H_2^d+\dots +(1+a_{d+1})H_{d+1}^d+ H_{d+2}^d+\dots + H_t^d$$
against the minimality of $t$.
\end{rem}

In our notation $Z\subset\PP^3=\PP(\CC[x_0,x_1,y_0,y_1])$ and the subsets 
$Z(F)$ and $Z(G)$ span two skew lines, namely the lines corresponding to
$\PP(\CC[x_0,x_1])$ and $\PP(\CC[y_0,y_1])$. We will denote these lines with
$R_F\supset Z(F)$ and $R_G\supset Z(G)$.

\begin{lem}\label{noRG} There is no line $R$ containing $d$ points of $Z(H)$ 
and intersecting $R_G$. 
\end{lem}
\begin{proof} Assume by contradiction that such a line $R$ exists, say 
$R\supset \{H_1,\dots, H_{d}\}$. 
The projection $\pi$ from $R_G$ to $R_F$ is obtained by setting $y_0=y_1=0$. 
Since $\pi$  contracts $R$ to a point,
it follows that $F$ is generated by the $d$-th powers of the images of
$H_1$ and $H_{d+1},\dots, H_t$. So the rank of $F$ is at most  $t-d+1$.
On the other hand $t-d+1<r+s-d+1\leq r$, since we are assuming $s\leq d-1$. 
Hence we have a contradiction.
\end{proof}

We now prove SAC for two variables.

\begin{thm}\label{add2}  Let $d>1$. If $F\in\mathbb{C}[x_0,x_1]$ and $G\in\mathbb{C}[y_0, y_1]$ are non-zero degree $d$ forms, then
$$\rk(F+G)=\rk(F)+\rk(G).$$
 \end{thm}
 \begin{proof} We will use all the previous notation and we may assume what follows:
 \begin{itemize}
 \item$3\leq r\leq d$, $3\leq s\leq d-1$, and thus $d\geq 4$ (see Remark \ref{step0} 
 and Proposition \ref{maxcas}).
 \item  $Z(H)$ does not intersect $R_F\cup R_G$ (see Lemma \ref{noH}).
 \item $Z(H)$ spans $\PP^3$ (see Lemma \ref{spanP3}).
 \item  At most $d$ points of $Z(H)$ lie on a line (see Remark \ref{dline}).
 \end{itemize}
 
By contradiction assume that $t<r+s$ and thus $t\leq 2d-2$. 
We let
\[Z=Z(F)\cup Z(G)\cup Z(H)\]
and we will find the contradiction $h^1\ii_Z(d)=0$ (see Remark \ref{h1rem}).

\noindent{\bf Case 1}:  there is no line containing $d$ points of $Z(H)$. 
Lemma \ref{celine}, with $u=d$, yields that $h^1\ii_{Z(H)}(d-2)=0$.
Since $Z(G)$ has cardinality at most $d-1$, we apply Lemma \ref{resid}
with $u=d-1$, and obtain $h^1\ii_{Z(H)\cup Z(G)}(d-1)=0$. Finally, since 
$Z(F)$ has cardinality at most $d$, we may apply Lemma \ref{resid}
with $u=d$, and obtain $h^1\ii_Z(d)=0$,
a contradiction.
 
\noindent{\bf Case 2}: there are $d$ points of $Z(H)$ on a line $R$. By  Lemma \ref{noRG},
 $R$ cannot intersect $R_G$. $R$ may or may not intersect $R_F$ and, in the first case,
 the point $R\cap R_F$ may or may not lie in $Z(F)$. In any case, set
 $\Phi=R\cap R_F$ and set $Z'=Z(H)\setminus R$, $Z''=(Z(H)\cap R) \cup \Phi$.
 Notice that $Z'$ contains at most $d-2$ points and $Z''$ contains at most $d+1$ points. 
 Notice also that $Z$ is contained in $Z(F)\cup Z(G)\cup Z''\cup Z'$, with equality
 if and only if either $\Phi=\emptyset$ or $\Phi$ is not a point of $Z(F)$.
 
 Since the cardinality of $Z'$ is at most $d-2$, then $h^1\ii_{Z'}(d-3)=0$.
Because $Z(G)$ has cardinality at most $d-1$, we apply Lemma \ref{resid}
with $u=d-2$, and obtain $h^1\ii_{Z'\cup Z(G)}(d-2)=0$. 
Since $Z(F)\setminus \Phi$ has cardinality at most $d$, we may apply Lemma \ref{resid}
with $u=d-1$, and obtain $h^1\ii_{Z'\cup Z(G)\cup (Z(F)\setminus \Phi)}(d-1)=0$.
Finally, notice that $Z'\cup Z(G)\cup (Z(F)\setminus \Phi)$ cannot intersect $R$.
 Thus, since $Z''$ has cardinality at most $d+1$, we apply Lemma \ref{resid}
with $u=d$, and we obtain 
$$ h^1\ii_{Z'\cup Z(G)\cup (Z(F)\setminus \Phi)\cup Z''}(d)=0.$$
Since $Z\subseteq {Z'\cup Z(G)\cup (Z(F)\setminus \Phi)\cup Z''}$ we get the contradiction
\[h^1\ii_Z(d)\leq h^1\ii_{Z'\cup Z(G)\cup (Z(F)\setminus \Phi)\cup Z''}(d)=0.\]
   \end{proof}

Theorem \ref{mainthm} is now proved by collecting the previous results.
\begin{proof}[{Proof of Theorem \ref{mainthm}}]
The cases $n=0$ and $m=0$ follow from Proposition \ref {Marvi} and Theorem \ref{add2} yields the rest of the statement. 
\end{proof}

\end{document}